\documentclass[ pdflatex]{sn-jnl}

\usepackage{amsmath,amssymb,amsthm}

\newtheorem{theorem}{Theorem}

\newtheorem{corollary}[theorem]{Corollary}
\begin{document}

\title{A Note on Carmichael Numbers of the form $ N=(d+1)(3d+1)(4d+1)$.}


\author{\fnm{Hayat} \sur{Bensella}} \email{hbensella@usthb.dz / bensella.hayet@gmail.com}

\affil{\orgdiv{Faculty of Mathematics, LATN Laboratory}, \orgname{USTHB}, \orgaddress{\street{BP 32, El Alia}, \city{Bab Azzouar}, \postcode{16111}, \state{Algiers}, \country{Algeria}}}

\abstract{We study integers of the form

\[
N=(d+1)(3d+1)(4d+1).
\]
Using Korselt's criterion, we prove that, provided the three
factors are prime, $N$ is a Carmichael number if and only if
\[
d\equiv4\pmod{12}.
\]

We also show that the special case $d=2^m$ produces only
$1105$ under the prime-factor hypothesis.
}


\maketitle

\section{Introduction}
A Carmichael number is a composite positive integer $N$ satisfying
\[
a^{N-1}\equiv 1\pmod N
\]
for every integer $a$ such that $\gcd(a,N)=1$. A fundamental
characterization of these numbers is given by Korselt's criterion \cite{Korselt1899}:
a composite integer $N$ is a Carmichael number if and only if
$N$ is squarefree and
\[
p-1\mid N-1
\]
for every prime divisor $p$ of $N$.

The study of Carmichael numbers has a long history. In particular,
Alford, Granville, and Pomerance proved that infinitely many
Carmichael numbers exist~\cite{Alford1994}.
This breakthrough stimulated further investigations into their
distribution, arithmetic structure, and occurrence in special
families of integers. Banks, Finch, Luca, Pomerance, and
St\u{a}nic\u{a} studied connections between Sierpiński and
Carmichael numbers~\cite{Banks2015}, while Cilleruelo, Luca, and
Pizarro investigated Carmichael numbers in sequences of the form
$(2^n k+1)_{n\geq1}$~\cite{Cilleruelo2016}. 
Other works have explored related questions concerning the density and
distribution of arithmetic sets \cite{Erdos1979}, the
distribution of primes in short intervals
\cite{Granville2015} and the impossibility of certain
types of Carmichael numbers \cite{Wright2012}.
An important classical construction of Carmichael numbers is due to Chernick, who considered integers of the form
\[
N=(6k+1)(12k+1)(18k+1).
\]
Whenever the three factors are prime, their product is a Carmichael number. This construction illustrates how Korselt's criterion can be used to identify Carmichael numbers through products of linear polynomials. Motivated by this approach, we consider the family
\[
N=(d+1)(3d+1)(4d+1),
\]
our main result gives a necessary and sufficient condition for
this product to be a Carmichael number, provided that all three
factors are prime. Specifically, we prove that
\[
N\text{ is a Carmichael number}
\quad\Longleftrightarrow\quad
d\equiv4\pmod{12}.
\]
We then consider the special case $d=2^m$ and show that, under
the same primality assumption, the only Carmichael number arising
from this specialization is $1105$.

\section{Main theorem and proof}

\begin{theorem}\label{thm:main}
Let $d$ be a positive even integer such that
\[
d+1,\qquad 3d+1,\qquad 4d+1
\]
are all prime. Then
\[
N=(d+1)(3d+1)(4d+1)
\]
is a Carmichael number if and only if
\[
d\equiv4\pmod{12}.
\]
\end{theorem}

\begin{proof}
Suppose first that $N$ is a Carmichael number. By assumption,
the three factors
\[
p_1=d+1,\qquad p_2=3d+1,\qquad p_3=4d+1
\]
are prime. They are pairwise distinct, so $N$ is squarefree.

Expanding the product, we obtain
\[
\begin{aligned}
N
&=(d+1)(3d+1)(4d+1)\\
&=12d^3+19d^2+8d+1.
\end{aligned}
\]
Therefore,
\[
N-1=d(12d^2+19d+8).
\]

By Korselt's criterion,
\[
p_i-1\mid N-1,\qquad i=1,2,3.
\]

Since $p_1-1=d$, the first divisibility condition is automatic.

For the second prime factor, we have $p_2-1=3d$. Thus
\[
3d\mid N-1,
\]
which is equivalent to
\[
3\mid12d^2+19d+8.
\]
Reducing modulo $3$, we obtain
\[
d+2\equiv0\pmod3,
\]
and hence
\[
d\equiv1\pmod3.
\]

For the third prime factor, $p_3-1=4d$. Consequently,
\[
4d\mid N-1,
\]
which is equivalent to
\[
4\mid12d^2+19d+8.
\]
Reducing modulo $4$, we get
\[
3d\equiv0\pmod4.
\]
Since $\gcd(3,4)=1$, it follows that
\[
d\equiv0\pmod4.
\]

Combining
\[
d\equiv1\pmod3
\qquad\text{and}\qquad
d\equiv0\pmod4,
\]
the Chinese remainder theorem gives
\[
d\equiv4\pmod{12}.
\]

Conversely, suppose that $d\equiv4\pmod{12}$ and that all
three factors are prime. Then
\[
d\equiv1\pmod3,\qquad d\equiv0\pmod4.
\]
It follows that
\[
3\mid12d^2+19d+8
\]
and
\[
4\mid12d^2+19d+8.
\]
Therefore,
\[
d\mid N-1,\qquad
3d\mid N-1,\qquad
4d\mid N-1.
\]
Equivalently,
\[
p_i-1\mid N-1,\qquad i=1,2,3.
\]

The three prime factors are distinct, so $N$ is squarefree.
It is also composite. Korselt's criterion now implies that
$N$ is a Carmichael number.
\end{proof}

\subsection{Examples}

\begin{table}[htbp]
\centering
\caption{Examples of Carmichael numbers in the family
$N=(d+1)(3d+1)(4d+1)$.}
\label{tab:examples}
\begin{tabular}{c|c|c|c|c}
\hline
$d$ & $d+1$ & $3d+1$ & $4d+1$ & $N$ \\
\hline
$4$   & $5$    & $13$   & $17$   & $1105$ \\
$112$ & $113$  & $337$  & $449$  & $17098369$ \\
$400$ & $401$  & $1201$ & $1601$ & $771043201$ \\
\hline
\end{tabular}
\end{table}

\subsection{The special case $d=2^m$}

We now consider  
\[
d=2^m,\qquad m\geq1.
\]
In this case,
\[
N=(2^m+1)(3\cdot2^m+1)(2^{m+2}+1).
\]

\begin{corollary}\label{cor:power-two}
Let $m\geq1$ and suppose that
\[
2^m+1,\qquad
3\cdot2^m+1,\qquad
2^{m+2}+1
\]
are all prime. Then $N$ is a Carmichael number if and only if
$m=2$. In this case,
\[
N=1105.
\]
\end{corollary}

\begin{proof}
By Theorem~\ref{thm:main}, the Carmichael property is equivalent
to
\[
2^m\equiv4\pmod{12}.
\]
This congruence requires $m$ to be even and $m\geq2$.

Suppose that $m>2$. Since $m$ is even, either
\[
m\equiv2\pmod4
\qquad\text{or}\qquad
m\equiv0\pmod4.
\]

If $m\equiv2\pmod4$, then
\[
2^m\equiv4\pmod5,
\]
so
\[
5\mid2^m+1.
\]
Because $m>2$, the integer $2^m+1$ is greater than $5$,
and hence it is composite. This contradicts the hypothesis.

If $m\equiv0\pmod4$, then $m+2\equiv2\pmod4$, so
\[
2^{m+2}\equiv4\pmod5.
\]
Thus
\[
5\mid2^{m+2}+1.
\]
Again, this factor is greater than $5$ and is composite,
a contradiction.

Therefore $m=2$. Conversely, when $m=2$, the three factors
are $5$, $13$, and $17$, and their product $1105$ is a
Carmichael number by Theorem~\ref{thm:main}.
\end{proof}

\section{Further questions}

Theorem~\ref{thm:main} gives a complete characterization of
the Carmichael numbers in the family
\[
N=(d+1)(3d+1)(4d+1).
\]
The remaining arithmetic question is to determine how many
positive integers $d\equiv4\pmod{12}$ make all three numbers
\[
d+1,\qquad 3d+1,\qquad 4d+1
\]
simultaneously prime.

The theorem provides a necessary and sufficient condition for
each individual value of $d$, but it does not prove that
infinitely many such values exist. Establishing an infinitude
result, or obtaining further information about the distribution
of these prime triples, would require additional arguments.

\end{document}